\documentclass{amsart}
\usepackage{fullpage}
\usepackage{stylefile}

\title[Dobrushin's ergodicity coefficient for Markov operators]{Dobrushin's ergodicity coefficient for Markov operators on cones}

\author{St\'ephane Gaubert}
\address{INRIA and CMAP UMR 7641 CNRS\\
         \'Ecole Polytechnique\\
          91128 Palaiseau C\'edex, France}
\email[]{Stephane.Gaubert@inria.fr}
\thanks{The authors were partially supported by the PGMO Programme of FMJH and EDF, and by the programme ``Ing{\'e}nierie Num{\'e}rique \& S{\'e}curit{\'e}'' of the French National Agency of Research, project ``MALTHY'', number ANR-13-INSE-0003.}
\thanks{An announcement of some of the present results has appeared in the Proceedings of the ECC'13 (European Control Conference), July 17-18 2013, Zurich.}
\author{Zheng Qu}
\address{School of Mathematics\\ University of Edinburgh \\Edinburgh, EH9 3FD, UK}
\email[]{zheng.qu@ed.ac.uk}
\date{October 4, 2014}
\begin{document}

\begin{abstract}
Doeblin and Dobrushin characterized the contraction rate of Markov operators with respect the total variation norm. We generalize their results by
giving an explicit formula for the contraction rate of a Markov operator over a cone in terms of pairs of extreme points with disjoint support in a set of abstract probability measures. By duality, we derive a characterization of the contraction
rate of consensus dynamics over a cone with respect to Hopf's oscillation
seminorm (the infinitesimal seminorm associated with Hilbert's projective metric). 
We apply these results to Kraus maps
(noncommutative Markov chains, representing quantum channels),
and characterize the ultimate contraction of the map
in terms of the existence of a rank one matrix
in a certain subspace.
\end{abstract}

\maketitle

\section{Introduction}
A basic result in the theory of Markov chains, due to Doeblin and Dobrushin,
is the characterization of the
contraction rate of a Markov operator 
acting on a space of measures equipped 
with the total variation norm. 
Consider in particular a finite Markov chain
with transition (row stochastic) matrix $A \in \R^{n\times n}$.
The associated Markov operator is the map $\nu\mapsto \nu A$
from $\R^n$ to $\R^n$, where the elements of $\R^n$ are thought of as row vectors.
The set of probability measures can be identified to the
standard simplex $\mathcal{P}:= \{\nu\in \R^n\mid \nu_i \geq 0 \text{ for } 1\leq i\leq n,\; \sum_{1\leq i\leq n}\nu_i=1\}$, and the total variation norm is nothing but one half of the $\ell_1$ norm $\|\cdot\|_1$ on $\R^n$. 
We are interested in the Lipschitz constant of the map $\nu \mapsto \nu A, \; \mathcal{P}\to \mathcal{P}$ with respect to the total variation norm,  
or equivalently, in 
the operator norm of the map $\nu \mapsto \nu A$ on 
the subspace of vectors of $\R^n$ with zero sum, equipped
with the same norm,
\begin{align*}
\delta(A) &:=
\max_{\nu,\pi\in \mathcal{P}, \; \nu\neq \pi } \frac{\|\nu A-\pi A\|_1}{\|\nu -\pi\|_1}
= \max_{\scriptstyle \nu \in \R^n, \; \|\nu\|_1=1 \atop  \scriptstyle \sum_{1\leq i\leq n}\nu_i =0} \|\nu A\|_1 \enspace.
\end{align*}
The Doeblin-Dobrushin characterization reads
\begin{align}
\delta(A)&=  \frac 1 2 \max_{i<j} \sum_{1\leq s\leq n}|A_{is}-A_{js}| \enspace ,\label{e-dd}
\\
&= 1-\min_{i<j}\sum_{s=1}^n \min(A_{is},A_{js}) \enspace. \label{e-dd2}
\end{align}
The expression of $\delta(A)$ given by~\eqref{e-dd} is known
as {\em Doeblin contraction coefficient}, see~\cite{PeresLevin}, whereas the
second expression, in~\eqref{e-dd2}, 
is known as 
{\em Dobrushin ergodicity coefficient}~\cite{Dobrushin56}.
The latter is often used to show that $\delta(A)<1$. This holds
in particular if there is a {\em Doeblin state}, i.e., a distinguished
state $t$ such that $A_{it}\geq \epsilon>0$ for all $1\leq i\leq n$.
Then, $\delta(A)\leq 1-\epsilon$.

A dual characterization of $\delta(A)$ has been used
in linear consensus theory. The latter is motivated
by communication networks, control theory and parallel computation~\cite{Hirsch89,Bertsekas89,Boydrand06,Moreau05,Blondel05convergencein,OlshevskyTsitsiklis,AngeliBliman}. It studies dynamics of the
form
\begin{align}\label{a-xkintroAk}
x_{k+1}=A_kx_k,\enspace k=1,2,\dots,\qquad x_k\in \R^n
\end{align}
where $A_k$ are row stochastic matrices, acting
on column vectors. One looks for conditions which guarantee
the convergence of $x_k$ to a {\em consensus state}, i.e.,
to a scalar multiple of the unit vector $e$ of $\R^n$. To this end,
one considers the following seminorm,
sometimes called {\em diameter} or {\em Tsitsiklis' Lyapunov function}~\cite{Tsitsiklis86}
\[
\Delta(x)=\max_{1\leq i,j\leq n} (x_i-x_j) \enspace ,
\]
for all $x,y\in \R^n$. It is known~\cite{SpeilmanMorse} that 
\begin{align}
\delta(A) = \max_{x \not \in \R e} \frac{\Delta(Ax)}{\Delta(x)}
\label{e-tsitsiklis}
\end{align}
so that the Doeblin-Dobrushin ergodicity coefficients
coincides with the one-step contraction rate of the consensus dynamics
with respect to the diameter seminorm. We note that the
same seminorm $\Delta$ is a fundamental tool in Perron-Frobenius
theory, where it is called {\em Hopf's oscillation}~\cite{hopf,Bushell73} 
or {\em Hilbert's seminorm}~\cite{arxiv1}. 

In this paper, we extend the Doeblin-Dobrushin theorem, as well
as the dual characterization~\eqref{e-tsitsiklis}, to 
Markov operators over cones. We consider a bounded linear self-map $T$ 
of a Banach space $\cX$, equipped with
a normal cone $\C \subset \cX$, and a {\em unit} element 
$\unit$ belonging to the interior of $\C$.
We say that $T$ is an {\em abstract Markov operator}
if $T(\C)\subset \C$ and $T(\unit)=\unit$.
The {\em Hopf oscillation} in
the space $\cX$ is the seminorm defined by
$$
x\mapsto \hilbert{x}{\unit}:=\inf\{\beta-\alpha: \alpha \unit \preceq x\preceq \beta \unit\} \enspace ,
$$
where $\preceq$ denotes the partial order induced by $\C$.
Our main result reads:
\begin{theo}[Contraction rate in Hopf's oscillation seminorm]\label{th-intro1}
 Let $T:\cX \to \cX$ 
be an abstract Markov operator.
Then
\begin{align*}
\sup_{\substack{z\in \cX\\ \hilbert{z}{\unit}\neq 0}}
\frac{\hilbert{T(z)}{\unit}}{\hilbert{z}{\unit}}
 &= \sup_{\substack{\nu,\pi\in \pP(\unit)\\ \nu\neq \pi}} \frac{\othernorm{T^\star(\nu-\pi)}_T^\star}{\othernorm{\nu-\pi}_{T}^\star}
\\
&=\frac 1 2 \sup_{\substack{\nu,\pi \in \operatorname{extr} \cP(\unit)\\ \nu\perp\pi}} 
\|T^\star(\nu)-T^\star(\pi)\|_T^\star\\
&= 1-\inf_{\substack{\nu,\pi \in \operatorname{extr}\cP(\unit)\\ \nu\perp\pi}}\inf_{x\in[0, \unit]} \<\pi,T(x)>+\<\nu, T(\unit-x)> \enspace .
\end{align*}
\end{theo}
This theorem follows from Theorems~\ref{th-opnorm} and 
~\ref{theo-markTH} below. 
The notations and notions used here are
detailed in Section~\ref{sec-operatornormch}.
In particular, $T^\star$ denotes the adjoint of $T$ and we make use of the following norm, which we call {\em Thompson's norm},
\[ \|z\|_{T} = \inf\{\alpha>0:\, -\alpha \unit\preceq z\preceq \alpha\unit\}\]
on the space $\cX$, and denote by $\|\cdot\|_T^\star$ the dual norm. The notation
$\mathcal{P}(\unit)=\{\mu\in \C^\star :\, \<\mu,\unit> =1\}$ refers to the
abstract {\em simplex} of the dual Banach space $\cX^\star$ of $\cX$,
where $\C^\star$ is the  dual cone of $\C$;
$\extr$ denotes the extreme points of a set; $\bot$ denotes a certain
{\em disjointness} relation, which will be seen to generalize
the condition that two measures have disjoint supports.

Taking $\cX=\R^n$, $\C$ the standard positive cone $\R^n_+$,
and $\unit$ the standard unit vector $(1,\dots,1)^\top$, we recover from Theorem~\ref{th-intro1}
the Doeblin-Dobrushin
characterization~\eqref{e-dd},\eqref{e-dd2}, as well as
its dual form in linear consensus theory~\eqref{e-tsitsiklis}.

Results related to Theorem~\ref{th-intro1} have previously
appeared. In a finite dimensional setting, 
Reeb, Kastoryano, and Wolf
~\cite{ReebWolf2011} gave a characterization analogous 
to the second equality of the above theorem without the disjointness condition. We refer to Remark~\ref{rem-comReeb} for a comparison. Also, Mukhamedov
gave in~\cite{mukhamedov},
in the setting of von Neumann algebras, a
characterization similar to the same equality, still without the disjointness
condition. He established some other properties
of the ergodicity coefficient,
and derived
ergodic type theorems for nonhomogeneous Markov chains.

Several motivations lead to consider Markov operators over
cones which differ from the standard positive cone of $\R^n$. 

First, Sepulchre, Sarlette, and Rouchon~\cite{sepulchre} and independently, 
Reeb, Kastoryano and Wolf~\cite{ReebWolf2011} ,
have shown
that tools from Perron-Frobenius theory (specially contraction
results in different metrics over cones) provide a unifying general
approach to address issues from quantum information and control.
Here, quantum channels are represented 
by self-maps $T$ of the cone of positive semidefinite matrices, 
preserving the Loewner order, and the identity matrix. Relations with
classical ``consensus'' theory were also addressed in~\cite{sepulchre}.
We derive further results,
showing that Theorem~\ref{th-intro1} leads to a
%
noncommutative analogue of Dobrushin's ergodicity coefficient (see Corollary~\ref{coro-noncommDobr}):
$$
1-
\displaystyle\min_{\substack{X=(x_1,\dots,x_n)\\XX^*=I_n}}\min_{\substack{u,v:u^*v=0\\u^*u=v^*v=1}} \sum_{i=1}^n \min\{u^*T(x_ix_i^*)u,v^*T(x_ix_i^*)v\}\enspace.
$$
Then, we use the above formula to show that the convergence of a noncommutative consensus system or equivalently the ergodicity
of a noncommutative Markov chain can be characterized by the existence of a rank one matrix in  certain subspace of matrices (Theorem~\ref{th-globalconv}
and~\ref{th-markerg}). Also, it follows from these results that 
an operator $T$ representing a quantum channel
has a contraction rate of $1$ (absence of contraction)
with respect to Hopf's oscillation if and only if
there exists two distinguishable pure states, 
i.e., a quantum clique of cardinality $2$~\cite{Shor},
or equivalently if the quantum channel has a positive zero-error
capacity~\cite{Medeiros}.

We also derive as a direct illustration
a convergence result (geometric convergence
of the iterates of the operator to a rank one operator, or
geometric convergence to a ``consensus state'') 
in Theorem~\ref{th-ex-con}.
Actually, the present contraction results are useful more generally
when considering iterates of random contractions. Then,
almost sure convergence to a consensus state
can be obtained by adapting ideas of Bougerol~\cite{Bougerol93},
see the discussion in~\S\ref{sec-consensusoperator} below.
We limited our convergence treatment here to simple illustrations
of our results: we note that the question of proving  ``weak ergodicity
results'' in their best generality has been thoroughly studied,
we refer the reader to the work of Mukhamedov~\cite{mukhamedov},
and to the references therein.


Our second and original motivation arises from
{\em non-linear}, rather than linear, Perron-Frobenius theory, i.e.,
from the study of non-linear maps over cones. In this setting,
the interior of a cone $\C$ is equipped with {\em Hilbert's projective metric},
defined by:
$$
d_H(x,y):=\log\inf\{\frac{\beta}{\alpha}:\; \alpha,\beta>0,\; \alpha x\preceq y\preceq \beta x\}.
$$
Birkhoff~\cite{birkhoff57} characterized the
contraction ratio with respect to $d_H$ of a linear map $T$ preserving the interior $\C^0$ of the cone $\C$, 
$$
\sup_{x,y\in \C^0}\frac{d_H(Tx,Ty)}{d_H(x,y)}=\tanh(\frac{\Diam T(\C^0)}{4}),
$$
$$\Diam T(\C^0):=\sup_{x,y\in \C^0} d_H(Tx,Ty)\enspace.
$$
This fundamental result, which implies that a linear map sending the cone $\C$ into its interior is a strict contraction in Hilbert's metric, can be used
to derive the 
  Perron-Frobenius theorem
from
the Banach contraction mapping theorem, see~\cite{Bushell73,Kohlberg82,EvesonNuss95} for more information.

The generalization of Birkhoff's theorem to non-linear
maps, and in particular, the computation of the Lipschitz
constant of nonlinear maps with respect to Hilbert's
projective metric, has important applications
(including population dynamics), and it has motivated several works,
specially the one of Nussbaum~\cite{nussbaum94},
who observed
that $d_H$ is the weak Finsler metric obtained when taking $\hilbert{\cdot}{\unit}$ to be the infinitesimal distance at point $\unit$. In other words,
\[
d_H(x,y) = \inf_{\gamma} \int_0^1 \hilbert{\dot{\gamma}(s)}{\gamma(s)} ds 
\]
where the infimum is taken over piecewise $C^1$ paths $\gamma:[0,1]\to \C^0$ 
such that $\gamma(0)=x$ and $\gamma(1)=y$. 
He deduced that the contraction ratio,
with respect to Hilbert's projective metric,
 of a nonlinear map $f
:\C^0\rightarrow \C^0$ that is positively homogeneous of degree $1$
(i.e.\ $f(\lambda x)=\lambda f(x)$ for all $\lambda>0$), 
restricted to a geodesically convex subset $U\subset \C^0$,
can be expressed in terms of the Lipschitz constants of the linear maps $Df(x)$ with respect to a family of Hopf's oscillation seminorms: 
\begin{align}
\sup_{x,y\in U} \frac{d_H(f(x),f(y))}{d_H(x,y)}=\sup_{x\in U}
\kappa(x)\label{e-nussbaum}
\end{align}
where
\begin{align}
\kappa(x):=
\sup_{\substack{z\in \cX \\ \hilbert{z}{x}\neq 0 }} \frac{\hilbert{Df(x)z}{f(x)}}{\hilbert{z}{x}} \enspace .
\label{e-deltaa}
\end{align}
We recognize in $\kappa(x)$ a variant of the Doeblin-Dobrushin coefficient $\delta(A)$, in which the domain and range of $A$ are equipped
with different unit elements, namely $x$ and $f(x)$. Our characterization
carries over to this case. In particular, Theorem~\ref{th-opnorm} below
gives an explicit formula for $\kappa(x)$,
which, in combination with Nussbaum's characterization~\eqref{e-nussbaum}
allows one to compute the contraction rate of a non-linear map
with respect to Hilbert's projective metric.

\section{Thompson's norm and Hilbert's seminorm }\label{sec-pre}
We start by some preliminary results.
  Throughout the paper, $(\cX, \othernorm{\cdot})$ is a real Banach space. Denote by $\cX^\star$  the dual space of $\cX$. For any $x\in \cX$ and $q\in \cX^\star$, denote by $\<q,x>$ the value of $q(x)$.
Let $\C\subset \cX$ be a closed pointed convex cone
with nonempty interior $\C_0$
, in particular, $\alpha \C\subset \C$ for $\alpha \in \R^{+}$, $\C+\C\subset \C$ and $\C\cap (-\C)=0$. 
 The partial order $\preceq $ 
induced by $\C$ on $\cX$ is defined as follows:
\[
x\preceq y \Leftrightarrow y-x \in \C  \enspace.
\]
For $x\preceq y$ we define the order interval:
$$
[x,y]:=\{z\in \cX|x\preceq z\preceq y\}.
$$
For $x\in \cX$ and $y\in \C_0$, following~\cite{nussbaum88}, we define
\begin{align}\label{a-eq4}
\begin{array}{l}
M(x/y):=\inf\{t\in \R:x\preceq ty\}\\
m(x/y):=\sup\{t\in \R:x\succeq ty\}
\end{array}
\end{align}
Observe
that since $y\in\mathcal{C}_0$, and since $\mathcal{C}$ is closed
and pointed, the two sets in~\eqref{a-eq4} are non-empty, closed, and bounded
from below and from above, respectively. In particular, $m$ and $M$ take
finite values.

For $x\in \cX$ and $y\in \mathcal{C}_0$, we call \firstdef{oscillation}~\cite{Bushell73} the difference between $M(x/y)$ and $m(x/y)$:
$$
\hilbert{x}{y}:=M(x/y)-m(x/y).
$$
Let $\unit$ 
denote a distinguished element in $\C_0$, which we shall
call a {\em unit}. For $x\in \cX$, define
\[
\|x\|_{T}:=\max( M(x/\unit), -m(x/\unit))
\]\index{norm!Thompson's norm}\index{seminorm!Hilbert's seminorm}
which we call {\em Thompson's norm}, with respect to the element $\unit$, and
\[
\|x\|_H:= \hilbert{x}{\unit}
\]
which we call \firstdef{Hilbert's seminorm} with respect to the element $\unit$.   \index{unit element}
\begin{rem}These terminologies are motivated by the fact that Thompson's part metric and Hilbert's projective metric are Finsler metrics 
for which the infinitesimal distances at the point $\unit\in \C^ 0$ are respectively given by $\othernorm{\cdot}_T$ and $\othernorm{\cdot}_H$,
see~\cite{nussbaum94}.
The seminorm $\|\cdot\|_H$ is also called Hopf's oscillation seminorm~\cite{Bushell73}.  Besides, it is clear that the unit $\unit$ is an \firstdef{order unit}  and Thompson's norm $\|\cdot\|_{T}$ is the corresponding \firstdef{order unit norm}, see~\cite{Ellis, Alfsen,Nagel74}.
\end{rem}

We assume that the cone $\cX$ is  normal, that is, there is a constant $K>0$ such that
$$
0\preceq x\preceq y\Rightarrow \othernorm{x}\leq K\othernorm{y}.
$$
It is known that under this assumption the two norms $\othernorm{\cdot}$ and $\othernorm{\cdot}_T$ are equivalent, see~\cite{nussbaum94}. 
Therefore the space $\cX$ equipped with the norm $\othernorm{\cdot}_T$ is an order unit Banach space. Since Thompson's norm $\othernorm{\cdot}_T$ is defined
with respect to a particular element $\unit$, we write $(\cX,\unit,\othernorm{\cdot}_T)$ instead of $(\cX,\othernorm{\cdot}_T)$.
\index{metric!Finsler metric}

\begin{example}\label{rem1}
 We consider the finite dimensional vector space $\cX=\R^n$, the standard positive cone $\C=\R^n_+$ and the unit vector $\unit=\mathbf 1:=(1,\dots,1)^T$. 
It can be checked that 
 Thompson's norm with respect to $\mathbf 1$ is nothing but the sup norm
$$
\othernorm{x}_T=\max_i |x_i|=\othernorm{x}_{\infty},$$
whereas Hilbert's seminorm with respect to $\mathbf 1$ is the so called\index{diameter}
{\em diameter}:
$$
\othernorm{x}_H=\max_{1\leq i,j\leq n} (x_i-x_j)
=\Delta(x).
$$
\end{example}

\begin{example}\label{ex-SymnHilbert}
 Let $\cX=\sym_n$, the space of Hermitian matrices of dimension $n$ and $\C=\sym_n^+$, the cone of positive semidefinite matrices. Let the identity matrix  $I_n$ be  the unit element: $\unit=I_n$. Then Thompson's norm with respect to $I_n$ is nothing but the sup norm of the spectrum
of $X$, i.e.,
$$
\othernorm{X}_T=\max_{1\leq i\leq n} \lambda_i(X)=\othernorm{\lambda(X)}_{\infty},
$$
where $\lambda(X):=(\lambda_1(X),\dots,\lambda_n(X))$, is the vector of ordered eigenvalues of $X$, counted
with multiplicities,
whereas Hilbert's seminorm with respect to $I_n$ is the diameter of the spectrum:
$$
\othernorm{X}_H=\max_{1\leq i,j\leq n} (\lambda_i(X)-\lambda_j(X))=\Delta(\lambda(X)).
$$
\end{example}


\section{Abstract simplex in the dual space and dual unit ball}\label{sec-abstract simplex}
We denote by $(\cX^\star,\unit,\othernorm{\cdot}_T^\star)$ the dual space of $(\cX,\unit,\othernorm{\cdot}_T)$ where
 the dual norm $\othernorm{\cdot}_T^\star$ of a continuous linear functional $z \in \cX^\star$ is defined by:
$$\othernorm{z}_T^\star:=\sup_{\othernorm{x}_T=1} \<z,x>\enspace.$$
The \firstdef{abstract simplex} in the dual space is defined by:\index{abstract simplex}
\begin{align}\label{a-simp}\cP(\unit):=\{\mu \in \C^\star\mid \<\mu, \unit> =1\}\enspace,\end{align}
where $\C^\star$ is the {\em dual cone} of $\C$: $$\C^\star=\{z\in \cX^\star:\braket{z, x}\geq0 \enspace \forall x\in \C\}\enspace.$$ 


\begin{rem}\label{rem-rn1}
For  the standard positive cone (Example~\ref{rem1}, $\cX=\R^n$, $\C=\R^n_+$ and $\unit=\mathbf 1$), the dual space $\cX^\star$ is $\cX=\R^n$ itself and the dual norm $\othernorm{\cdot}_T^\star$ is the $\ell_1$ norm:
$$
\othernorm{x}_T^\star=\sum_i |x_i|=\othernorm{x}_1.
$$ 
The abstract simplex $\cP(\mathbf 1)$ is the standard simplex in $\R^n$:
$$
\cP(\mathbf 1)=\{\nu\in \R^n_+: \sum_i \nu_i=1\},
$$
i.e., the set of probability measures on the discrete space $\{1,\dots,n\}$.
\end{rem}
\begin{rem}\label{rem-sym1}
For the cone of semidefinite matrices (Example~\ref{ex-SymnHilbert}, $\cX=\sym_n$, $\C=\sym_n^+$ and $\unit=I_n$), the dual space $\cX^\star$ is $\cX=\sym_n$ itself and the dual norm $\othernorm{\cdot}_T^\star$ is the trace norm:
$$
\othernorm{X}_T^\star=\sum_{1\leq i\leq n} |\lambda_i(X)|=\othernorm{X}_1,\quad X\in \sym_n
$$ 
The simplex $\cP(I_n)$  is the set of positive semidefinite matrices with trace $1$:
$$
\cP(I_n)=\{\rho\in \sym^n_+: \trace(\rho)=1\}.
$$
The elements of this set are called {\em density matrices} in quantum\index{matrix!density matrix}
physics. They are thought of as noncommutative analogues
of probability measures.
\end{rem}

By the duality between order unit and base normed spaces~\cite{Ellis}, the space $(\cX^\star,\unit,\othernorm{\cdot}_T^\star)$ is a base normed space. The abstract simplex $\cP(\unit)$ coincides with  the \firstdef{base} and the dual norm $\othernorm{\cdot}_T^\star$ with the \firstdef{base norm}.  
We denote by $B_T^ \star(\unit)$ the  dual unit ball:
$$
B_T^ \star(\unit)=\{x\in \cX^ \star\mid \othernorm{x}_T^\star\leq 1\}\enspace.
$$
We denote by $\operatorname{conv}(S)$ the convex hull of a set $S$.
The next lemma relates the abstract simplex $\cP(\unit)$ to the dual unit ball $B_T^\star(\unit)$. The proof can be found in~\cite{Ellis,Alfsen}. 
\begin{lemma}[\cite{Ellis}]\label{l-BTstar}
The dual unit ball $B_T^\star(\unit)$
of the space $(\cX^\star,\unit,\othernorm{\cdot}_T^\star)$, satisfies
\begin{align}
B_T^\star(\unit) = \operatorname{conv}(\cP(\unit)\cup -\cP(\unit))\enspace.
\label{e-bstar}
\end{align}
\end{lemma}
\begin{rem}
 Reeb, Kastoryano, and Wolf~\cite{ReebWolf2011} defined a \firstdef{base} $\mathcal{B}$ of a proper cone $\mathcal{K}$ in a finite dimensional vector space $\mathcal{V}$,  which coincides with the definition of our ``abstract simplex''. They defined the \firstdef{base norm} of $\mu\in \mathcal{ V} $ with respect to $\mathcal{B}$ by:
$$
\othernorm{\mu}_{\mathcal{B}}=\inf\{\lambda\geq 0: \mu\in \lambda \operatorname{conv}(\mathcal{B}\cup-\mathcal{B})\}.
$$
They also defined the \firstdef{distinguishability norm} of $\mu\in \mathcal{V}$ by:
\begin{align}\label{a-vxunit}
\othernorm{\mu}_{\tilde M}=\sup_{0\preceq x\preceq \unit} \<\mu,2x-\unit>.
\end{align}
And Theorem 14 in their paper~\cite{ReebWolf2011} states that the distinguishability norm is equal to the base norm:
\begin{align}\label{a-dualReebWolf}
\othernorm{\mu}_{\tilde M}=\othernorm{\mu}_{\mathcal{B}} \enspace .
\end{align}
Lemma~\ref{l-BTstar} is equivalent to the duality result~\eqref{a-dualReebWolf} of Reeb et al., in a finite dimensional setting, and their approach can be seen as a dual one to ours.
\end{rem}


\section{Characterization of extreme points of the dual unit ball}\label{sec-characdualunitball}
A standard result of functional analysis shows that if
$\mathcal{W}$ is a closed subspace of a Banach space $(\mathcal{X},\othernorm{\cdot})$, then the quotient space $\cX/\mathcal{W}$ is a Banach space,
canonically equipped with the {\em quotient norm}
\[
x\mapsto \inf_{w\in \cW} \othernorm{x+ w} \enspace ,
\]
see~\cite[Chap.~III, \S~4]{Conway90FA}.
The next lemma shows that when $\cX$ is equipped
with Thomspon's norm, Hilbert's seminorm coincides with the
quotient norm of $\cX/\R\unit$,
up to a factor $2$.
\begin{lemma}\label{l-quonorm}
For all $x\in \cX$, we have:
\[
\|x\|_H = 2 \inf_{\lambda\in \R}\|x+\lambda \unit\|_T 
\]
\end{lemma}
\begin{proof}
The expression $$\|x+\lambda \unit\|_T = \max (M(x/\unit)+\lambda,
-m(x/\unit)-\lambda)$$ is minimal when
$M(x/\unit)+\lambda=
-m(x/\unit)-\lambda$. Substituting the value of $\lambda$ obtained
in this way in $\|x+\lambda \unit\|_T$, we arrive at the announced formula.
\end{proof}
\begin{lemma}\label{l-norm-rel}
The quotient normed space $(\cX/\R\unit, \othernorm{\cdot}_H)$ is a Banach space. Its dual is
 $(\cM(\unit), \othernorm{\cdot}_H^\star)$ where 
$$
\cM(\unit):=\{\mu \in \cX^\star | \<\mu, \unit>=0\},
$$
and
\begin{align}\label{a-hs-ts}
\othernorm{\mu}_H^\star:=\frac{1}{2} \othernorm{\mu}_T^\star, \enspace \forall \mu \in \cM(\unit).
\end{align}
\end{lemma}
\begin{proof}
It is shown in~\cite[Chap.~III, Theorem 10.2]{Conway90FA} that
if
$\mathcal{W}$ is a closed subspace of a Banach space $(\mathcal{X},\othernorm{\cdot})$, the dual of the quotient space $\mathcal{X}/\mathcal{W}$ can be identified
isometrically to the space of continuous linear forms on $\mathcal{X}$ that vanish on $\mathcal{W}$, equipped with the dual norm $\othernorm{\cdot}^\star$ of $\mathcal{X}^\star$.
Specializing this result to the case in which $\cX$ is equipped
with twice the Thompson norm and $\mathcal{W}=\R\unit$, and noting that
multiplying the norm on the space $\cX$ by a given positive factor
divites the corresponding dual norm by the same factor, 
we obtain~\eqref{a-hs-ts}.
\end{proof}
\if
The Hilbert's seminorm defines a norm in the quotient space $\cX/\R\unit$. 
It is easily seen from the definitions that
\begin{align}\label{a-nHT}
\othernorm{x}_H\leq 2\othernorm{x}_T.
\end{align}
Thus we know that the dual space of $(\cX/\R\unit, \othernorm{\cdot}_H)$ is a subset of $\cX^\star$.  It is immediate that the dual space should be contained in the hyperplane:
\[
\cM = \{\mu\in \cX^\star\mid \<\mu,\unit> =0\}.
\]
The next lemma shows that the dual space is exactly the hyperplane.
\begin{lemma}\label{l-norm-rel}
The dual space of $(\cX/\R\unit, \othernorm{\cdot}_H)$ is $(\cM, \othernorm{\cdot}_H^\star)$ where 
\begin{align}\label{a-hs-ts}
\othernorm{\mu}_H^\star:=\frac{1}{2} \othernorm{\mu}_T^\star, \enspace \forall \mu \in \cM.
\end{align}
\end{lemma}
\begin{proof}
Let any $x\in \cX/\R \unit$ and 
$$
\lambda=-\frac{M(x/\unit)+m(x/\unit)}{2},
$$
then 
$$
\|x+\lambda \unit\|_T=\frac{1}{2}\othernorm{x}_H.
$$
Let any $\mu \in \cM$, then
\[
\<\mu,x> =\<\mu,x+\lambda \unit>\leq \|\mu\|_T^\star\|x+\lambda \unit\|_T= \frac{1}{2}\|\mu\|_T^\star \|x\|_H.
\]
Hence the hyperplane $\cM$ is the dual space of $(\cX/\R\unit, \othernorm{\cdot}_H)$ and the dual norm $\othernorm{\cdot}_H^\star$ should satisfy:
$$\|\mu\|_H^\star \leq \frac 1 2 \|\mu\|_T^\star.$$
Finally it follows from the relation~\eqref{a-nHT} that:
$$
\|\mu\|_H^\star \geq \frac 1 2 \|\mu\|_T^\star.
$$
\end{proof}
\fi
The above lemma implies that the unit ball  of the space $(\cM(\unit), \othernorm{\cdot}_H^\star)$, denoted by
$B_H^\star(\unit)$, satisfies:
\begin{align}\label{a-BHstar}
B_H^\star(\unit)=2B_T^\star(\unit) \cap \cM(\unit).
\end{align}
\begin{rem}
 In the case of the standard positive cone ($\cX=\R^n$, $\C=\R^n_+$ and $\unit=\mathbf 1$),
Lemma~\ref{l-norm-rel} implies that for any two probability measures $\mu,\nu\in \pP(\mathbf 1)$, the dual norm $\othernorm{\mu-\nu}_H^\star$ is the total variation distance between $\mu$ and $\nu$:
$$
\othernorm{\mu-\nu}_H^\star=\frac{1}{2}\othernorm{\mu-\nu}_1=\othernorm{\mu-\nu}_{TV}
$$
\end{rem}

Before giving a representation of the extreme points of $B_H^\star(\unit)$,
 we define a \firstdef{disjointness} relation $\bot$ on $\pP(\unit)$.\index{disjointness}
\begin{defi}
 For all $\nu,\pi \in \pP(\unit)$, we say that $\nu$ and $\pi$ are {\em disjoint}, denoted by $\nu \perp \pi$, if 
$$
\mu =\frac{\nu+\pi}{2}
$$
for all $\mu \in \pP(\unit)$ such that $\mu \succeq \frac{\nu}{2}$ and $\mu\succeq \frac{\pi}{2}$. 
\end{defi}
\begin{example}
In the case of the standard positive cone ($\cX=\R^n$, $\C=\R^n_+$ and $\unit=\mathbf 1$), 
two points $\nu,\pi$ in  $\pP(\mathbf 1)$ are disjoint if and only if
for all $1\leq i\leq n$, $\nu_i=0$ or $\pi_i=0$ holds, meaning
that $\nu$ and $\pi$, thought of as discrete probability measures,
have disjoint supports.
\end{example}
We have the following characterization of the disjointness property.
\begin{lemma}\label{l-nuperppi}
 Let $\nu,\pi\in \pP(\unit)$. The following assertions are equivalent:
\begin{itemize}
\item [(a)]$\nu\perp\pi$.
\item[(b)] The only elements $\rho,\sigma\in \pP(\unit)$ satisfying
$$
\nu-\pi=\rho-\sigma
$$
are $\rho=\nu$ and $\sigma=\pi$.
\end{itemize}
\end{lemma}
\begin{proof}
(a)$\Rightarrow$ (b): Let any $\rho,\sigma\in \pP(\unit)$ such that
$$
\nu-\pi=\rho-\sigma.
$$
Then it is immediate that 
$$
\nu+\sigma=\pi+\rho. 
$$
Let $\mu=\frac{\nu+\sigma}{2}=\frac{\pi+\rho}{2}$. Then $\mu\in\pP(\unit)$, $\mu\succeq \frac{\nu}{2}$ and $\mu\succeq \frac{\pi}{2}$. Since $\nu \perp\pi$, we obtain that $\mu=\frac{\nu+\pi}{2}$. It follows that $\rho=\nu$ and $\sigma=\pi$.

(b)$\Rightarrow$ (a): Let any $\mu\in \pP(\unit)$ such that $\mu\succeq \frac{\nu}{2}$ and $\mu\succeq \frac{\pi}{2}$. Then 
$$
\nu-\pi=(2\mu-\pi)-(2\mu-\nu).
$$
From (b) we know that $2\mu-\pi=\nu$.
\end{proof}
We denote by $\extr(\cdot)$ the set of extreme points of a convex set.
\begin{prop}\label{p-carac-ext}
The set of extreme points of $B_H^\star(\unit)$, denoted by $\extr{B_H^\star(\unit)}$, is characterized by:
$$
\extr{B_H^\star(\unit)}=\{\nu-\pi\mid\nu,\pi\in \extr{\pP(\unit)},\nu \perp \pi\}.
$$ 
\end{prop}
\begin{proof}
It follows from~\eqref{e-bstar} that every
point $\mu\in B_T^\star(\unit)$ can be written
as $$\mu= s\nu - t\pi$$ with $ s+t=1,s,t\geq 0$,
$\nu,\pi \in \cP(\unit)$. Moreover, if $\mu\in \cM(\unit)$, then
$$0=\<\mu,\unit>=s\<\nu,\unit>-t\<\pi,\unit>
= s-t,$$thus $s=t=\frac{1}{2}$. 
Therefore every $\mu \in B_T^\star(\unit) \cap \cM(\unit)$ can be written as
\[
\mu = \frac{\nu-\pi}{2}, \enspace \nu,\pi \in \pP(\unit).
\]
Therefore by~\eqref{a-BHstar} we proved that
\begin{align}\label{a-BHstarunit}
B_H^\star(\unit)=\{\nu-\pi:\nu,\pi\in \pP(\unit)\}.
\end{align}
Now let $\nu,\pi \in \extr \pP(\unit)$ and $\nu\perp \pi$. We are going to prove that $\nu-\pi\in \extr B_H^\star(\unit)$. Let $\nu_1,\pi_1,\nu_2,\pi_2\in \pP(\unit)$ such that
$$
\nu-\pi=\frac{\nu_1-\pi_1}{2}+\frac{\nu_2-\pi_2}{2}.
$$
Then 
$$
\nu-\pi=\frac{\nu_1+\nu_2}{2}-\frac{\pi_1+\pi_2}{2}.
$$
By Lemma~\ref{l-nuperppi}, the only possibility is $2\nu={\nu_1+\nu_2}$ and $2\pi=\pi_1+\pi_2$.
Since $\nu,\pi\in\extr \pP(\unit)$ we obtain that $\nu_1=\nu_2=\nu$ and $\pi_1=\pi_2=\pi$. Therefore $\nu-\pi\in \extr{B_H^\star(\unit)}$.

Now let $ \nu,\pi\in \pP(\unit)$ such that $\nu-\pi \in \extr{B_H^\star(\unit)}$. 
Assume by contradiction that $\nu$ is not extreme in $\pP(\unit)$ (the case in which
$\pi$ is not extreme can be dealt with similarly). Then, we can
find $\nu_1,\nu_2\in \pP(\unit)$, $\nu_1\neq \nu_2$, such that
$\nu = \frac{\nu_1+ \nu_2}{2}$. It follows that $$\mu = \frac{\nu_1-\pi}{2}
+\frac{\nu_2-\pi}{2},$$ where ${\nu_1-\pi},\nu_2-\pi$ are distinct
elements of $ B_H^\star(\unit)$, which is a contradiction. Next we show that $\nu\perp \pi$. To this end,
let any $\rho,\sigma\in \pP(\unit)$ such that 
$$
\nu-\pi=\rho-\sigma.
$$
Then 
$$
\nu-\pi=\frac{\nu-\pi+\rho-\sigma}{2}=\frac{\nu-\sigma}{2}+\frac{\rho-\pi}{2}.
$$
If $\sigma\neq \pi$, then $\nu-\sigma\neq \nu-\pi$ and this contradicts the fact that $\nu-\pi$ is extremal. Therefore $\sigma=\pi$ and $\rho=\nu$.
From Lemma~\ref{l-nuperppi}, we deduce that $\nu\perp\pi$.


\end{proof}
\begin{rem}
In the case of standard positive cone ($\cX=\R^n$, $\C=\R^n_+$ and $\unit=\mathbf 1$),  the set of extreme points of $\pP(\mathbf 1)$ is the set of 
standard basis vectors $\{e_i\}_{i=1,\dots,n}$. The extreme points are pairwise disjoint.
\end{rem}
\begin{rem}\label{rem-xx*yy*}
In the case of cone of semidefinite matrices ($\cX=\sym_n$, $\C=\sym_n^+$ and $\unit=I_n$), the set of extreme points of $\pP(I_n)$ is
$$
\extr \pP(I_n)=\{xx^*\mid x\in \CC^n, x^*x=1\}\enspace,
$$
which are called \firstdef{pure states} in quantum information terminology.\index{pure state}
Two extreme points $xx^*$ and $yy^*$ are disjoint if and only if $x^*y=0$. To see this, note that if $x^*y=0$ then any Hermitian matrix $X$ such that $X\succeq xx^*$ and $X\succeq yy^*$ should satisfy
$X\succeq xx^ *+yy^ *$. Hence by definition $xx^*$ and $yy^ *$ are disjoint. Inversely, suppose that $xx^ *$ and $yy^ *$ are disjoint and consider the spectral decomposition of the matrix $xx^*-yy^*$, i.e., there is $\lambda\leq 1$ and two orthonormal vectors $u,v$ such that $xx^*-yy^*=\lambda(uu^ *-vv^ *)$. It follows that $xx^*-yy^*=uu^ *-((1-\lambda)uu^ *+\lambda vv^ *)$.
By Lemma~\ref{l-nuperppi}, the only possibility is $yy^ *=(1-\lambda)uu^ *+\lambda vv^ *$ and $xx^ *=uu^*$ thus $\lambda=1$, $u=x$ and $v=y$. Therefore $x^ *y=0$.
\end{rem}

\section{The operator norm induced by Hopf's oscillation seminorm}\label{sec-operatornormch}
Consider two real Banach spaces $\cX_1$ and $\cX_2$. Let $\C_1\subset \cX_1$ and $\C_2\subset \cX_2$ be respectively 
two closed pointed convex normal cones with non empty interiors $\C_1^ 0$ and $\C_2^0$.
Let $\unit_1\in \C_1^ 0$ and $\unit_2\in \C_2^ 0$. 
Then, we know from Section~\ref{sec-characdualunitball} that the two quotient spaces $(\cX_1/\R \unit_1,\othernorm{\cdot}_H)$ and $(\cX_2/\R \unit_2,\othernorm{\cdot}_H)$  are Banach spaces. The dual spaces of $(\cX_1/\R \unit_1,\othernorm{\cdot}_H)$ and $(\cX_2/\R \unit_2,\othernorm{\cdot}_H)$ are respectively  $(\cM(\unit_1),\othernorm{\cdot}_H^ \star)$ and $(\cM(\unit_2),\othernorm{\cdot}_H^ \star)$ (see Lemma~\ref{l-norm-rel}).

Let $T$ be a continuous
linear map from the space $(\cX_1/\R \unit_1,\othernorm{\cdot}_H)$ to
$(\cX_2/\R \unit_2,\othernorm{\cdot}_H)$.
The operator norm of $T$, denoted by $\othernorm{T}_H$, is given by:
\begin{align}\label{e-def-opH}
\|T\|_H:= \sup_{\othernorm{x}_H=1} \|T(x)\|_H =\sup \frac{\omega(T(x)/\unit_2)}{\omega(x/\unit_1)}\enspace.
\end{align}
 By definition, the \firstdef{adjoint operator} $T^\star: (\cM(\unit_2),\othernorm{\cdot}_H^ \star) \rightarrow (\cM(\unit_1),\othernorm{\cdot}_H^ \star) $ of $T$ is:
$$
\<T^\star(\mu),x>=\<\mu,T(x)>,\enspace \forall \mu \in \cM(\unit_2),x\in \cX_1/\R \unit_1.
$$
The operator norm  of $T^\star$, denoted by $\othernorm{T^ \star}_H^ \star$, is then:
\[ \|T^\star\|_H^\star:= \sup_{\mu \in B_H^\star(\unit_2)} \|T^\star(\mu)\|^\star_H.
\]
A classical duality result (see \cite[\S~6.8]{aliprantis}) shows that an operator and its adjoint have the same operator norm. In particular, 
$$
\|T\|_H = \|T^\star\|_H^\star.
$$
\begin{theo}\label{th-opnorm}
Let $T:\cX_1 \to \cX_2$ 
be a bounded linear map such that $T(\unit_1)\in \mathbb{R} \unit_2$. Then,
\begin{align*}
\|T\|_H  &=\|T^\star\|_H^\star= \frac 1 2 \sup_{\substack{\nu,\pi \in  \pP(\unit_2)}} 
\|T^\star(\nu)-T^\star(\pi)\|_T^\star\\
&=\sup_{\substack{\nu,\pi \in \pP(\unit_2)}} \sup_{ x\in[0, \unit_1]}\<\nu-\pi,T(x)>.
\end{align*}
Moreover, the supremum 
can be restricted to the set 
of mutually disjoint extreme points:
\begin{align}
\|T\|_H  =\|T^\star\|_H^\star&= \frac 1 2 \sup_{\substack{\nu,\pi \in \operatorname{extr} \cP(\unit_2)\\ \nu\perp\pi}} 
\|T^\star(\nu)-T^\star(\pi)\|_T^\star\nonumber\\
&=\sup_{\substack{\nu,\pi \in \operatorname{extr}\cP(\unit_2)\\ \nu\perp\pi}} \sup_{ x\in[0, \unit_1]}\<\nu-\pi, T(x)>.
\label{e-carac-TH}
\end{align}
\end{theo}
\begin{proof}
 We already noted that $\|T\|_H = \|T^\star\|_H^\star$. 
Moreover, 
\[
\|T^\star\|_H^\star
= \sup_{\mu\in B_H^\star(\unit_2)} \|T^\star(\mu)\|_H^\star.
\]
By the characterization of $B_H^\star(\unit_2)$ in~\eqref{a-BHstarunit}
and the characterization of the norm $\|\cdot\|_H^\star$ in Lemma~\ref{l-norm-rel}, we get
\begin{align*}
\sup_{\mu\in B_H^\star(\unit_2)} \|T^\star(\mu)\|_H^\star
&=\sup_{\nu,\pi\in \pP(\unit_2)} \|T^\star(\nu)-T^\star(\pi)\|_H^\star
\\ &=\frac 1 2 \sup_{\nu,\pi\in \pP(\unit_2)} \|T^\star(\nu)-T^\star(\pi)\|_T^\star
\enspace .
\end{align*}
For the second equality, note that
\begin{align*}
\othernorm{T^\star(\nu)-T^\star(\pi)}_T^\star
&=\displaystyle\sup_{x\in[0,\unit_1]} \<T^\star(\nu)-T^\star(\pi),2x-\unit_1>\\
&=2\displaystyle\sup_{ x\in[0,\unit_1]} \<T^\star(\nu)-T^\star(\pi),x>\enspace.
\end{align*}
We next show that the supremum can be restricted to the set of extreme points. By the Banach-Alaoglu theorem, $B_H^\star(\unit_2)$ is weak-star compact,
and it is obviously convex. The dual space $\cM(\unit_2)$ endowed with the weak-star topology is a locally convex topological space. Thus by the Krein-Milman theorem, the unit ball $B_H^\star(\unit_2)$, which is a compact convex set in 
$\cM(\unit_2)$ with respect to the weak-star topology, is the closed convex hull of its extreme points. So every element $\rho$ of $B_H^\star(\unit_2)$
is the limit of a net  $(\rho_\alpha)_\alpha$ of elements in $\operatorname{conv}\big(\operatorname{extr} B_H^\star(\unit_2)\big)$.
Observe now that the function
\[ \varphi: \mu \mapsto \|T^\star(\mu)\|_H^\star = \sup_{x\in B_H(\unit_1)} \<T^\star(\mu),x>
=\sup_{x\in B_H(\unit_1)} \<\mu, T(x)>
\]
which is a sup of weak-star continuous maps is convex and  weak-star lower
semi-continuous. This implies that 
\begin{align*}\varphi(\rho)
&\leq  \liminf_\alpha \varphi(\rho_\alpha) \\&\leq \sup\{\varphi(\mu): \mu\in \operatorname{conv}\big(\operatorname{extr} B_H^\star(\unit_2)\big)\} \\& = \sup\{\varphi(\mu): \mu\in \operatorname{extr} B_H^\star(\unit_2)\}\enspace.
\end{align*}
Using the characterization of the extreme points in Proposition~\ref{p-carac-ext}, we get:
\begin{align*}
\sup_{\mu\in B_H^\star(\unit_2)} \|T^\star(\mu)\|_H^\star &= 
\sup_{\mu\in\operatorname{extr} B_H^\star(\unit_2)} \|T^\star(\mu)\|_H^\star 
\\
&=\sup_{\substack{\nu,\pi \in \operatorname{extr}\cP(\unit_2)\\ \nu\perp\pi}} \|T^\star(\nu)-T^\star(\pi)\|_H^ \star.\qedhere
\end{align*}
\end{proof}

\begin{rem}\label{rk-reached}
When $\cX_1$ is of finite dimension, the set $[0,\unit_1]$ is the convex hull of the set of its extreme points, hence, the supremum over the variable $x\in[0,\unit_1]$ in~\eqref{e-carac-TH} is attained at an extreme point. Similarly,
if $\cX_2$ is of finite dimension, the suprema over $(\nu,\pi)$ in the same
equation are also attained, because the map $\varphi$ in the proof of the previous theorem, which is a supremum of an equi-Lipschitz family of maps,
is continuous (in fact, Lipschitz).
\end{rem}
\begin{rem}\label{rem-comReeb}
 Theorem~\ref{th-opnorm} should be compared with a result of~\cite{ReebWolf2011} which can be stated as follows.
\begin{prop}[Proposition 12 in~\cite{ReebWolf2011}]\label{p-Reeb12}
Let $\cV,\cV'$ be two finite dimensional vector spaces and $L:\mathcal{V}\rightarrow \mathcal{V'}$ be a linear map and let $\mathcal{B}\subset \mathcal{V}$ and $\mathcal{B'}\subset \mathcal{V}'$ be bases. Then
\begin{align}\label{a-Reeb12}
\sup_{v_1\neq v_2\in\mathcal{B}} \frac{\othernorm{L(v_1)-L(v_2)}_{\mathcal{B'}}}{\othernorm{v_1-v_2}_{\mathcal{B}}}=\frac{1}{2} 
\sup_{v_1,v_2\in \extr \mathcal{B}}\othernorm{L(v_1)-L(v_2)}_{\mathcal{B}'}
\end{align}
\end{prop}
The first term in~\eqref{a-Reeb12} is called the \firstdef{contraction ratio} of the linear map $L$, with respect to base norms. 
One important application of this proposition concerns the \firstdef{base preserving} maps $L$ such that $L(\mathcal{B})\subset \mathcal{B'}$.
Let us translate this proposition in the present setting. Consider a linear map $T:\cX_1/\R\unit_1\rightarrow \cX_2/\R\unit_2$. Then $T^\star: \cX_2^ \star\rightarrow \cX_1^ \star$ is a base preserving linear map ($ T^ \star(\pP(\unit_2))\subset \pP(\unit_1)$) and so,
Proposition~12 of~\cite{ReebWolf2011} shows that:
\begin{align}
\|T^\star\|_H^\star=\sup_{\substack{\nu,\pi\in \pP(\unit_2)\\ \nu\neq \pi}} \frac{\othernorm{T^\star(\nu-\pi)}_T^\star}{\othernorm{\nu-\pi}_{T}^\star}=\frac{1}{2} \sup_{\nu,\pi
\in \extr \pP(\unit_2)} \othernorm{T^\star(\nu)-T^\star(\pi)}_{T}^\star
\label{e-disjoint}
\end{align}
Hence, by comparison with~\cite{ReebWolf2011}, the additional
information here is the equality between 
the contraction ratio in Hilbert's seminorm of a unit preserving linear map $\|T\|_H$, and the contraction ratio with respect to the base norms
of the dual base preserving map $\|T^\star\|_H^\star$. The latter is the primary
object of interest in quantum information theory 
whereas the former
is of interest in the control/consensus literature~\cite{Tsitsiklis86,Moreau05}.
We also proved that the supremum in~\eqref{e-disjoint} can
be restricted to pairs of {\em disjoint} extreme points $\nu,\pi$. Finally, the expression of the contraction rate as the last supremum 
in Theorem~\ref{th-opnorm} leads here to an abstract version of Dobrushin's ergodic coefficient, see Eqn~\eqref{e-dd2}
and Corollary~\ref{coro-noncommDobr} below.
\end{rem}
Let us recall the definition of Hilbert's projective metric.
\begin{defi}[\cite{birkhoff57}] \firstdef{Hilbert's projective metric} between two elements $x$ and $y$ of $\C_0$ is
\begin{align}
d_H(x,y)=\log(M(x/y)/m(x/y))
\enspace .
\label{e-def-hilbproj}
\end{align}
(The notation $M(\cdot/\cdot)$ was defined in~\eqref{a-eq4}.)
\end{defi}\index{metric!Hilbert's projective metric}
Consider a linear operator $T:\cX_1\rightarrow \cX_2$ such that $T(\C_1^0)\subset \C_2^0$.
Following~\cite{birkhoff57,Bushell73},  the \firstdef{projective diameter} of $T$ is defined as below:
$$
\Diam T=\sup \{d_H(T(x),T(y)): x,y\in \C_1^0\}.
$$\index{projective diameter}\index{oscillation ratio}
Birkhoff's contraction formula~\cite{birkhoff57,Bushell73} states that the oscillation ratio equals to the contraction ratio of $T$ and they are related to its projective diameter.
\begin{theo}[\cite{birkhoff57,Bushell73}]\label{th-birkho}
$$
\sup_{x,y\in \C_1^0} \frac{\omega(T(x)/T(y))}{\omega(x/y)}=
\sup_{x,y\in \C_1^0} \frac{d_H(T(x),T(y))}{d_H(x,y)}= \tanh(\frac{\Diam T}{4}).$$
\end{theo}
 Following~\cite{ReebWolf2011}, we define the projective diameter of $T^\star$:
$$
\Diam T^\star=\sup \{d_H(T^\star(u),T^\star(v)):  u,v\in \C_2^\star\backslash 0\}.
$$
Note that $\Diam T=\Diam T^\star$. This is because
$$
\begin{array}{ll}
\displaystyle\sup_{x,y\in\C_1^0}  \frac{M(T(x)/T(y))}{m(T(x)/T(y))}
&=\displaystyle\sup_{x,y\in\C_1^0} \sup_{u,v\in \C_2^\star\backslash 0} \frac{\<u,T(x)>\<v,T(y)>}{\<u,T(y)>\<v,T(x)>}\\
&=
\displaystyle\sup_{u,v\in \C_2^\star\backslash 0 } \frac{M(T^ \star(u)/T^ \star(v))}{m(T^ \star(u)/T^ \star(v))}
\end{array}
$$
\begin{coro}[Compare with~\cite{ReebWolf2011}]\label{th-TdiamT}
Let $T:\cX_1 \to \cX_2$ 
be a bounded linear map such that $T(\unit_1)\in \mathbb{R} \unit_2$ and $T(\C_1^0)\subset \C_2^0$ , then:
$$
\othernorm{T^ \star}_H^ \star=\othernorm{T}_H\leq \tanh(\frac{\Diam T}{4})=\tanh(\frac{\Diam T^ \star}{4})
$$
\end{coro}
\begin{proof}
 It is sufficient to prove the inequality. For this, note that
$$
\othernorm{T}_H=\sup_{\substack{x\in \cX_1\\ \omega(x/\unit_1)\neq 0}} \omega(T(x)/\unit_2)/\omega(x/\unit_1)=\sup_{\substack{x \in \C_1^0\\\omega(x/\unit_1)\neq 0}} \omega(T(x)/\unit_2)/\omega(x/\unit_1).
$$
Then we apply Birkhoff's contraction formula.
\end{proof}
\begin{rem}
Reeb et al.~\cite{ReebWolf2011} showed in a different way that
$$
\othernorm{T^\star}_H^\star\leq \tanh(\frac{\Diam T^ \star}{4})\enspace,
$$
in a finite dimensional setting.
Corollary~\ref{th-TdiamT} shows that as soon as the duality formula
$\othernorm{T^ \star}_H^ \star=\othernorm{T}_H$ is established, the latter inequality follows from Birkhoff's
contraction formula.
\end{rem}

\section{Application to Markov operators on cones and discrete time consensus dynamics}\label{sec-consensusoperator}
\if{
A classical result, which goes back to Doeblin and Dobrushin, characterizes
the Lipschitz constant of a Markov matrix acting on the space
of measures (i.e., a row stochastic matrix acting on the left), 
with respect to the total variation norm (see the discussion
in Section~\ref{sec-appclassiconsen} below). The same constant characterizes the contraction ratio with respect to the ``diameter'' (Hilbert's seminorm)
of the consensus system driven by this Markov matrix (i.e., a row
stochastic matrix acting on the right).
Markov operators on cones extend Markov matrices.
In this section, we extend to these abstract operators
a number of known properties of Markov matrices.}\fi

A bounded linear  map $T:\cX\to \cX$ is a {\em Markov operator}\index{operator!Markov operator} with respect to
a unit vector $\unit$ in the interior $\C^ 0$ of a closed convex pointed cone $\C\subset \cX$ if it satisfies the two following properties: 
\begin{itemize}
\item[(i)] $T$ is positive, i.e., $T(\C) \subset \C$.
\item[(ii)] $T$ preserves the unit element $\unit$, i.e., $T(\unit)=\unit$.
\end{itemize}

The case when $\othernorm{T}_H<1$ or equivalently $\othernorm{T^\star}_{H}^\star<1$
is of special interest; the following theorem
shows that the iterates of $T$ converge to a rank one projector with a rate bounded by $\|T\|_H$. 
\begin{theo}[Geometric convergence to consensus/invariant measure]\label{th-ex-con}
Let $T:\cX\rightarrow \cX$ be a Markov operator with respect to the unit element $\unit$.
If $\othernorm{T}_H<1$ or equivalently $\othernorm{T^\star}_{H}^\star<1$, then there is $\pi \in \pP(\unit)$ 
such that for all $x\in \cX$
\[
\|T^n(x)-\<\pi,x>\unit\|_T \leq (\othernorm{T}_H)^n \|x\|_H, 
\]
and for all $\mu \in \pP(\unit)$
\[
\|(T^\star)^n(\mu)-\pi \|_H^\star \leq (\othernorm{T}_H)^n.
\]
\end{theo}
\begin{proof}
The intersection $$\displaystyle\cap_{n} [m(T^n(x)/\unit),M(T^n(x)/\unit)]\subset \R$$ is
nonempty (as a non-increasing intersection of nonempty compact sets), and since $\othernorm{T}_H<1$ and
$$\omega(T^n(x)/\unit)\leq (\othernorm{T}_H)^n \omega(x/\unit),$$ this intersection must
be reduced to a real number $\{c(x)\}\subset \R$ depending on $x$, i.e.,
$$c(x)= \underset{n}{\cap} [m(T^n(x)/\unit),M(T^n(x)/\unit)]\enspace.$$
Thus for all $n\in \mathbb N$,
$$
-\omega(T^n(x)/\unit) \unit\leq T^n(x)-c(x)\unit\leq \omega(T^n(x)/\unit)\unit.
$$
Therefore by definition:
$$
\|T^n(x)-c(x)\unit \|_T\leq \omega(T^n(x)/\unit).
 \leq (\othernorm{T}_H)^n \othernorm{x}_H.
$$
It is immediate that:
$$
c(x)\unit=\lim_{n\rightarrow \infty} T^n(x)
$$
from which we deduce that $c:\cX\rightarrow \R$ is a continuous linear functional. Thus there is $\pi\in \cX^\star$ such that $c(x)=\<\pi,x>$.  Besides it is immediate that $\<\pi,\unit>=1$ and $\pi\in \C^\star$ because
$$
x\in \C \Rightarrow c(x)\unit \in \C\Rightarrow  c(x)\geq 0\Rightarrow \<\pi,x>\geq 0.
$$
Therefore $\pi\in \pP(\unit)$.
Finally for all $\mu \in \pP(\unit)$ and all $x\in \cX$ we have
$$
\begin{array}{ll}
\<(T^\star)^n(\mu)-\pi,x>&=\<\mu, T^n(x)-\<\pi,x>\unit>\\
&\leq \othernorm{\mu}_T^\star\othernorm{T^n(x)-\<\pi,x>\unit}_{T}\\
&\leq (\othernorm{T}_H)^n \|x\|_H.
\end{array}
$$
Hence
$$
\othernorm{(T^\star)^n(\mu)-\pi}_{H}^\star \leq (\othernorm{T}_H)^n.
$$
\end{proof}
A time invariant discrete time consensus system can be
described by
\begin{align}\label{a-xk+1Tk}
x_{k+1}=T(x_k),\quad k=1,2,\dots \enspace.
\end{align}
 The main concern of consensus theory is the convergence of the orbit $x_k$ to a 
{\em consensus state}, which is represented by a scalar multiple of the unit element $\unit$.
The dual system of~\eqref{a-xk+1Tk} represents a homogeneous discrete time Markov system:
\begin{align}\label{a-markovsystem}
 \pi_{k+1}=T^\star(\pi_k),\quad k=1,2,\dots\enspace.
\end{align}
One of the central issues in Markov chain study is the ergodic property, i.e., the convergence of the distribution $\pi_k$ to an {\em invariant measure}, given by a fixed point of $T^\star$.
Theorem~\ref{th-ex-con} shows that if $\othernorm{T}_H<1$ or equivalently $\othernorm{T^\star}_H^\star<1$, then the consensus system~\eqref{a-xk+1Tk} is globally convergent and the homogeneous Markov chain~\eqref{a-markovsystem} is  ergodic. 

 A time-dependent consensus system is described by
\begin{align}\label{a-tmedeconsesys}
x_{k+1}=T_{k+1}(x_k),\quad k=1,2,\dots
\end{align}
where $\{T_k:k\geq 1\}$ is a sequence of Markov operators sharing a common unit element $\unit\in \C^ 0$. Then if there is an integer $p>0$ and
a constant $\alpha<1$ such that for all 
$i\in \mathbb N$
$$
\othernorm{T_{i+p}\dots T_{i+1}}_H\leq \alpha,
$$
then the same lines of proof of Theorem~\ref{th-ex-con} imply the existence of $\pi \in \pP(\unit)$ such that for all $\{x_k\}$ satisfying~\eqref{a-tmedeconsesys},
$$
\othernorm{x_{k} -\<\pi, x_{0}>\unit}_T\leq \alpha^{\lfloor{\frac{k}{p}}\rfloor} \othernorm{x_{0}}_H,
\quad n\in \mathbb N.
$$
Moreover, if $\{T_k:k\geq 1\}$ is a stationary ergodic random process, then the almost sure 
convergence of the orbits of~\eqref{a-tmedeconsesys} to a consensus state
can be deduced by showing that $$\mathbb E[\log \othernorm{T_{1+p}\dots T_{1}}_H]<0$$ for some  $p>0$, see Bougerol~\cite{Bougerol93}.
The ergodicity of a inhomogeneous Markov chain can be studied in a dual approach. 
Hence,
in Markov chain and consensus applications, a central issue is  to compute
the operator norm $\|T\|_H$ of a Markov operator $T$.

A direct application of Theorem~\ref{th-opnorm} leads to following
characterization of the operator norm. 
\begin{theo}[Abstract Dobrushin's ergodicity coefficient]\label{theo-markTH}
 Let $T: \cX \to \cX$ be a Markov operator with respect to $\unit$. Then,
\[
\|T\|_H = \|T^\star\|_H^\star =1-\inf_{\substack{\nu,\pi \in \operatorname{extr}\cP(\unit)\\ \nu\perp\pi}}\inf_{x\in[0, \unit]} \<\pi,T(x)>+\<\nu, T(\unit-x)>.
\]
\end{theo}
\begin{proof}
Since $T(\unit)=\unit$, we have: 
\begin{align*}
&\sup_{\substack{\nu,\pi \in \operatorname{extr}\cP(\unit)\\ \nu\perp\pi}} \sup_{ x\in[0, \unit]}
\<\nu-\pi,T(x)>\\
&\qquad =\sup_{\substack{\nu,\pi \in \operatorname{extr}\cP(\unit)\\ \nu\perp\pi}}\sup_{x\in[0, \unit]} 1-\<\pi,T(x)>-\<\nu, T(\unit-x)>.
\end{align*}
\end{proof}
\begin{example}
Let us specialize Theorem~\ref{theo-markTH} to the case 
of the standard positive cone of $\R^n$ (Example~\ref{rem1}). Then, a
Markov operator $\R^n\to \R^n$ is of the form $T(x)=Ax$,
where $A$ is a row-stochastic matrix. 
We get 
\begin{align*}
\delta(A)=1-\min_{i< j} \min_{I\subset\{1,\dots,n\}}(\sum_{k\in I} A_{ik}+\sum_{k\notin I}A_{jk})\enspace.
\end{align*}
This formula yields directly the explicit form of Dobrushin's ergodicity coefficient recalled in the introduction, 
\[\delta(A)=
1-\displaystyle\min_{i< j}\sum_{s=1}^n \min(A_{is},A_{js}).
\]
\end{example}
\begin{rem}\label{rem-tmedepenran}
Several results of linear consensus theory can be interpreted,
or proved, in terms of ergodicity coefficient.
Consider the time-variant linear consensus system:
\begin{align}\label{a-xk+1Akxk}
x_{k+1}=A_kx_k,\enspace k=1,2,\dots\enspace,
\end{align}
where $\{A_k\}$ is a sequence of stochastic matrices. Moreau~\cite{Moreau05} showed that if all the non-zero entries of the matrices $\{A_k\}$ are bounded from below by a positive constant $\alpha>0$
and if there is $p\in \mathbb N$ such that for all $i\in \mathbb N$ there is a node connected to all other nodes in the graph associated to the matrix $A_{i+p}\dots A_{i+1}$, then the system~\eqref{a-xk+1Akxk} is globally uniformly convergent. These two conditions imply exactly that there is a Doeblin state associated to the matrix $A_{i+p}\dots A_{i+1}$. The uniform bound $\alpha$ is to have an upper bound on the contraction rate, more precisely,  $$\delta({A_{i+p}\dots A_{i+1}})\leq 1-\alpha,\enspace\forall i=1,2,\dots$$
\end{rem}

\section{Applications to noncommutative Markov operators}\label{sec-applnoncomm}
In this section, we specialize the previous general results to a finite dimensional noncommutative space ($\cX=\sym_n$, $\C=\sym_n^+$ and $\unit= I_n$, Example~\ref{ex-SymnHilbert}).

A completely positive unital linear map $\Phi:\sym_n\rightarrow \sym_n$ is characterized by a set of matrices $\{V_1,\dots,V_m\}$ satisfying 
\begin{align}\label{a-ViViIn}
\sum_{i=1}^m V_i^ *V_i=I_n\enspace
\end{align}
such that the map $\Phi$ is given by:
\begin{align}\label{a-PhiX}
\Phi(X)=\sum_{i=1}^m V_i^ *XV_i,\quad \forall X\in \sym_n\enspace.
\end{align}The matrices $\{V_1,\dots,V_m\}$ are called \firstdef{Kraus operators}.\index{Kraus operator} It is clear that $\Phi:\sym_n\rightarrow \sym_n$ defines a Markov operator.
The dual operator of $\Phi$ is given by:
$$
\Psi(X)=\sum_{i=1}^m V_i XV_i^*,\enspace X\in \sym_n\enspace.
$$
It is a completely positive and trace-preserving map, called Kraus map.\index{Kraus map} The map $\Phi$ and $\Psi$ represent a purely quantum channel~\cite{sepulchre,ReebWolf2011}.
In particular, the adjoint map $\Psi$ is trace-preserving and acts 
on density matrices.
The operator norm of $\Phi: \sym_n/\R I_n \rightarrow \sym_n/\R I_n$ is the contraction ratio with respect to the diameter of the spectrum:
$$
\othernorm{\Phi}_H=\sup_{X\in \sym_n} \frac{\lambda_{\max}(\Phi(X))-\lambda_{\min}(\Phi(X))}{\lambda_{\max}(X)-\lambda_{\min}(X)}.
$$
The operator norm of the adjoint map $\Psi: \pP(I_n)\rightarrow \pP(I_n)$ is the contraction ratio with respect to the trace norm (the total variation distance):
$$
\othernorm{\Psi}_H^\star=\sup_{\rho_1,\rho_2\in \pP(I_n)}\frac{\othernorm{\Psi(\rho_1)-\Psi(\rho_2)}_1}{\othernorm{\rho_1-\rho_2}_1} .
$$
The values $\othernorm{\Phi}_H$ and $\othernorm{\Psi}_H^\star$ are
the noncommutative counterparts of
$\delta(\cdot)$.

Specializing Theorem~\ref{theo-markTH} to  Kraus maps, we 
obtain the noncommutative version of Dobrushin's ergodicity coefficient. 
\begin{coro}[Noncommutative Dobrushin's ergodicity coefficient]\label{coro-noncommDobr}
Let $\Phi$ be a completely positive unital linear map defined in~\eqref{a-PhiX}. 
Then,
 \begin{align}\label{a-PhiHPsiH}
\othernorm{\Phi}_H=\othernorm{\Psi}_H^\star=1-
\displaystyle\min_{\substack{u,v:u^*v=0\\u^*u=v^*v=1}}\min_{\substack{X=(x_1,\dots,x_n)\\XX^*=I_n}} \sum_{i=1}^ n \min \{u^*\Phi(x_ix_i^ *)u,v^*\Phi(x_ix_i^ *)v\}
\end{align}
\end{coro}
\begin{proof}
It can be easily checked that 
$$
\extr [0,I_n]=\{P\in \sym_n:P^ 2=P\}.
$$
Hence, Theorem~\ref{theo-markTH} and Remark~\ref{rem-xx*yy*}
yield:
\begin{align*}
\othernorm{\Phi}_H&=\othernorm{\Psi}_H^\star=
1-\displaystyle\min_{\substack{u,v:u^*v=0\\u^*u=v^*v=1}}\min_{\substack{P^ 2=P}}u^*\Phi(I_n-P)u+v^*\Phi(P)v\\
&=1-\displaystyle\min_{\substack{u,v:u^*v=0\\u^*u=v^*v=1}}\min_{\substack{X=(x_1,\dots,x_n)\\XX^*=I_n}} \min_{J\subset\{1,\dots,n\}} \sum_{i\in J} u^*\Phi(x_ix_i^ *)u+ \sum_{i\notin J}v^*\Phi(x_ix_i^ *)v \\
\end{align*}
from which~\eqref{a-PhiHPsiH} follows.\qedhere
\end{proof}
\begin{rem}
For the noncommutative case, it is not evident whether more
%
effective characterization of the contraction rate exists. Note that the dual operator norm was studied in quantum information theory, see \cite{ReebWolf2011} and references therein. They provided a Birkhoff type upper bound (Corollary 9 in~\cite{ReebWolf2011}):
$$
\othernorm{\Psi}_{H}^ *\leq \tanh( \Diam \Psi/4)\enspace.
$$
The value $\Diam \Psi$ is not directly computable.
This upper bound is equal to 1 if and only if $\Diam \Psi=\infty$, which is satisfied if and only if there exist a pair of nonzero vectors $u,v\in \cC^n$  such that:
$$
\Span\{V_i u:1\leq i\leq m\}\neq \Span \{V_i v: 1\leq i\leq m\}.
$$
We next provide a
tighter, 
in fact necessary and sufficient, condition for the operator norm to be 1. 
\end{rem}
\begin{coro}\label{coro-nonco}
The following conditions are equivalent:
\begin{itemize}
\item[1.]$\othernorm{\Phi}_{H}=\othernorm{\Psi}_H^\star=1.$
\item[2.] There are nonzero vectors $u,v\in \cC^n$ such that
$$
\<V_i u,V_j v>=0, \enspace \forall i,j \in \{1,\dots,m\}.
$$
\item[3.] There is a  rank one matrix $Y\subset \cC^ {n\times n}$ such that
$$
\trace(V_i^*V_jY)=0,\enspace \forall i,j \in\{1,\cdots,m\}.
$$
\end{itemize}
\end{coro}
\begin{proof}
From Corollary~\ref{coro-noncommDobr} we know that $\othernorm{\Phi}_{H}=1$ if and only if there exist
an orthonormal basis $\{x_1,\dots,x_n\}$  and two vectors
$u,v \in \cC^ n$ of norm 1 such that
$$
\sum_{i=1}^ n \min \{\sum_{j=1}^ m u^*V_j^ *x_ix_i^ *V_ju,\sum_{j=1}^m v^*V_j^ *x_ix_i^ *V_jv\}=0\enspace.
$$
This is equivalent to that for each $i\in\{1,\dots,n\}$, either 
$$
x_i^*V_ju=0,\enspace \forall j=1,\dots,m
$$is true,
or
$$
x_i^*V_jv=0,\enspace \forall j=1,\dots,m
$$
is true. This is equivalent to 
$$
\<V_i u,V_j v>=0, \enspace \forall i,j \in \{1,\dots,m\}\enspace.
$$
The equivalence between the second and the third condition is trivial by taking $Y=vu^*$.
\end{proof}
\begin{rem}\label{quantumclique}
The condition
appearing in item~2 of Corollary~\ref{coro-nonco} is 
equivalent to the positivity of the zero-error
capacity~\cite{Medeiros}, or
to the existence of a quantum
clique of cardinality $2$~\cite{Shor}. The latter
problem is known to be QMA$_1$ complete (proof of Theorem~3.2 in~\cite{Shor}).
Thus, Corollary~\ref{coro-nonco} relates the absence of contraction
to a known hard problem in quantum computing.
\end{rem}


We consider a time-invariant noncommutative consensus system:
\begin{align}\label{a-sysnonc}
X_{k+1}=\Phi(X_k),\enspace k=1,2,\dots
\end{align}
where $\Phi$ is a completely positive unital map. To study the convergence of such system, Sepulchre, Sarlette and Rouchon~\cite{sepulchre}
proposed to study the contraction ratio
$$
\alpha:=\sup_{X\succ 0} d_H(\Phi(X),I_n)/d_H(X,I_n)\enspace.
$$
They applied Birkhoff's contraction formula (Theorem~\ref{th-birkho}) to give an upper bound on the contraction ratio $\alpha$:
$$
\alpha \leq \tanh( \Diam \Phi/4)\enspace.
$$
The following theorem is a direct corollary of Nussbaum~\cite{nussbaum94}.
\begin{theo}(Corollary of~\cite[Thm2.3]{nussbaum94})
$$
\othernorm{\Phi}_H=\lim_{\epsilon\rightarrow 0^+} \big(\sup \{\frac{d_H(\Phi(X),I_n)}{d_H(X,I_n)}:0<d_H(X,I_n)\leq \epsilon\}\big),
$$
\end{theo}
By this theorem, it is clear that the contraction ratio used in~\cite{sepulchre} is an upper bound of the operator norm $\| \Phi\|_H$:
$$
\othernorm{\Phi}_H\leq \alpha\enspace.
$$
We next provide an algebraic characterization of the global convergence of system~\eqref{a-sysnonc}, based on the result established in Corollary~\ref{coro-nonco}. 
Let us consider a sequence of matrix subspaces defined as follows:
\begin{align*}
& \cH_{0}=\Span\{I_n\}\enspace,\\
& \cH_{k+1}=\Span\{V_i^ * XV_j:X\in \cH_k, i,j=1,\dots,m\}\enspace,\enspace k=0,1,\dots,\enspace
\end{align*}
\begin{lemma}\label{l-Hkfini}
 There is $k_0\leq n^2-1$ such that 
$$
\cH_{k_0+s}=\cH_{k_0},\enspace \forall s\in \mathbb N.
$$
\end{lemma}
\begin{proof}
It follows from~\eqref{a-ViViIn} that $\cH_{k+1}\supseteq \cH_k$ for all $k\in \mathbb N$. 
Besides, if for some $k_0\in \bN$ such that
$$
\cH_{k_0+1}=\cH_{k_0}\enspace,
$$
then
$$
\cH_{k_0+s}=\cH_{k_0},\enspace \forall s\in \mathbb N.
$$
This property also implies that if for some $k_0\in \bN$
$$
\cH_{k_0+1}\neq \cH_{k_0}\enspace,
$$
then
$$
\cH_{k_0-s+1}\neq \cH_{k_0-s}\enspace,\forall 1\leq s\leq k_0\enspace.
$$
Since the dimension of $\cH_k$ can not exceed $n^2$, the case
$$
\cH_{k_0+1}\neq \cH_{k_0}\enspace,
$$
can not happen more than $n^2$ times. 
\end{proof}
For all $k\in \bN$, let $\cG_k$ be the orthogonal complement of $\cH_k$.
Then there is $k_0\leq n^ 2-1$ such that
\begin{align}\label{a-Gk}
\cG_{k}\supseteq \cG_{k+1},\enspace \forall k\in \bN;\enspace \enspace \cG_{k_0}=\cG_{k_0+s}, \enspace \forall s\in \mathbb{N}
\end{align}

\begin{theo}\label{th-globalconv}
The following conditions are equivalent:
\begin{itemize}
\item[(1)] There exists $k$ such that $\othernorm{\Phi^{k}}_{H}<1$.
\item[(2)] Every orbit of the system~\eqref{a-sysnonc} converges to an equilibrium co-linear to $I_n$.
\item[(3)] The subspace $ \displaystyle\cap_{k}\cG_{k}$ does not contain a rank one matrix.
\item[(4)] There exists $k_0\leq n^2-1$ such that $\othernorm{\Phi^{k_0}}_{H}<1$.
\end{itemize}
\end{theo}
\begin{proof}

$(1)\Rightarrow (2)$: We apply Theorem~\ref{th-ex-con} to the application $\Phi^ k$.

$(2)\Rightarrow (1)$:
Consider the quotient real linear space $\cW:= \sym_n/\R I_n$.
Since $\Phi(I_n)=I_n$, $\Phi$ yields a quotient linear map $\cW \mapsto \cW$.
We already observed in~\eqref{e-def-opH}
that $\othernorm{\Phi}_H$ is the operator
norm induced by the norm $\othernorm{\cdot}_H$ on $\cW$. 
It follows that $\othernorm{\Phi_1\Phi_2}_H\leq\othernorm{\Phi_1}_H
\othernorm{\Phi_2}_H$ holds for all linear maps $\Phi_1,\Phi_2: \cW\to\cW$,
and so, by Fekete's subadditive lemma,
\[
\inf_{k\geq 1} \othernorm{\Phi^ k}_H^ {1/k}=\lim_{k\rightarrow +\infty} \othernorm{\Phi^k}_H^ {1/k}
\]
Observe also that $\othernorm{\Phi}_H\leq 1$, so that
$\othernorm{\Phi^ k}_H \leq 1$ holds for all $k\geq 1$.
Then, if (1) is not true, we deduce
that
\begin{align}
\lim_{k\rightarrow +\infty} \othernorm{\Phi^k}_H^ {1/k}  = 1
 \enspace .
\label{e-fekete}
\end{align}
Now, for any real normed vector space $(\cV,\|\cdot\|_\cV)$, let $\cV_\CC= \cV+i\cV$
denote the complexification of $\cV$, and for any $\R$-linear self-map $T$
of $\cV$, let $T_\CC$ denote the complexification of $T$, so that
$T_\CC(x+iy)=T(x)+iT(y)$ for all $x,y\in\cV$. Recall that $\cV_\CC$ can
be equipped with the norm
\[
\|x+iy\|_{\cV_\CC} = \sup_{0\leq \theta\leq 2\pi} \|x\cos\theta -y\sin \theta\|_\cV
\]
and that the operator norm of $T$ induced by the norm $\|\cdot\|_\cV$ on $\cV$, 
denoted by $\|T\|_\cV$, as well as the operator norm of $T_\CC$ induced
by $\|\cdot\|_{\cV_\CC}$ on $\cV_\CC$, denoted by $\|T_\CC\|_{\cV_\CC}$, coincide,
\begin{align}
\|T_\CC\|_{\cV_\CC}=\|T\|_\cV \enspace .
\label{extension}
\end{align}
Consider in particular $(S_n)_\CC=S_n +iS_n\simeq \CC^{n\times n}$,
observe that $(S_n/\R I_n)_\CC \simeq \CC^{n\times n}/\CC I_n$.
It follows from~\eqref{extension} that
\[
\|\Phi_\CC^k \|_{\cW_\CC} = \|\Phi^k\|_{\cW} = \|\Phi^k\|_{H}
\]
holds for all $k$, and so, by~\eqref{e-fekete},
\begin{align}
\lim_{k\to \infty} \|\Phi_\CC^k \|^{1/k}_{\cW_\CC} = 1 \enspace .
\label{e-pregelfand}
\end{align}
By Gelfand's formula, the left-hand side of~\eqref{e-pregelfand}
is the spectral radius of the $\CC$-linear map
$\Phi_\CC:\cW_\CC \to \cW_{\CC}$. Hence, $\Phi_\CC$ has
an eigenvalue on the unit circle, meaning 
that there exists $\theta\in[0,2\pi)$, $X,Y \in S_n$, with $X+iY\not\in \CC I_n$, such that 
\[
\Phi(X+iY) - e^{i\theta} (X+iY) \in \CC I_n \enspace ,
\]
and so
\[
\Phi^k(X+iY) - e^{ik\theta}(X+iY)\in \CC I_n \enspace,
\]
for all $k\geq 1$. Identifying the real and imaginary parts, we get
$\Phi^k(X)= \cos (k\theta) X  -\sin(k\theta) Y + \alpha_k I_n$
and $\Phi^k(Y)= \sin(k\theta) X + \cos(k\theta) Y + \beta_k I_n$,
for some $\alpha _k, \beta_k\in \R$. 
Observe that since $X+iY \not \in \CC I_n$, we have $X,Y\not\in \R I_n$.
It follows that the orbit $(\Phi^k(X))_{k\geq 1}$ does not converge to a scalar multiple of $I_n$, contradicting (2).

$(3)\Leftrightarrow (1)$:
Note that for all $k\in \mathbb N$,
$$
\Phi^k(X)=\sum_{i_1,\dots,i_k} V_{i_k}^*\dots V_{i_1}^*XV_{i_1}\dots V_{i_k}.
$$
By Corollary~\ref{coro-nonco}, we know that $\othernorm{\Phi^k}_H=1$ if and only if the subspace $\cG_k$ contains a a rank one matrix.
Therefore , $\othernorm{\Phi^ k}_H=1$ for all $k\in \mathbb N$ if and only if the subspace $ \displaystyle\cap_{k}\cG_{k}$ contains a rank one matrix.

$(3)\Rightarrow (4)$:
By~\eqref{a-Gk}, there is $k_0 \leq n^2-1$ such that $\cG_{k_0}=\cap_{k}\cG_{k}$. It follows that if (3) is true then there is $k_0\leq n^ 2-1$ such that 
$\cG_{k_0}$ does not contain a rank one matrix. Then by Corollary~\ref{coro-nonco} we deduce that $\othernorm{\Phi^{k_0}}<1$ if (3) is true.
\end{proof}

In a dual way, the above analysis also applies to the ergodicity study of noncommutative Markov chain given by:
\begin{align}\label{a-markovnonco}
\Pi_{k+1}=\Psi(\Pi_k),\enspace k=1,2,\dots
\end{align}
Below is a dual version of Theorem~\ref{th-globalconv}.
\begin{theo}\label{th-markerg}
 The following conditions are equivalent:
\begin{itemize}
\item[(1)] There exists $k$ such that $\othernorm{\Psi^{k}}_{H}^\star<1$.
\item[(2)] The Markov chain~\eqref{a-markovnonco} converges to a unique invariant measure regardless of initial distribution.
\item[(3)] The subspace $ \displaystyle\cap_{k}\cG_{k}$ does not contain a rank one matrix.
\item[(4)] There exists $k_0\leq n^2-1$ such that $\othernorm{\Psi^{k_0}}_{H}^\star<1$.
\end{itemize}
\end{theo}


\begin{rem}
 A sufficient condition for the global convergence of the noncommutative consensus system~\eqref{a-sysnonc}
or equivalently, the ergodicity of the noncommutative Markov chain~\eqref{a-markovnonco} would be that there is $k_0\leq n^2-1$ such that
$$
\cH_{k_0}=\cC^ {n\times n}.
$$
Thus, checking the global convergence appears to be more tractable
than checking the one step contraction (compare this characterization
with the one of Corollary~\ref{coro-nonco}). 
\end{rem}

 \bibliographystyle{alpha}
\bibliography{biblio}
\end{document}